\documentclass[11pt]{article}
\usepackage{epsfig,amssymb,amsmath}

\newtheorem{theorem}{Theorem}
\newtheorem{corollary}[theorem]{Corollary}
\newtheorem{lemma}[theorem]{Lemma}
\newtheorem{proposition}[theorem]{Proposition}
\newtheorem{conjecture}[theorem]{Conjecture}
\newtheorem{problem}[theorem]{Problem}

\title{Strong embeddings of minimum genus}

\author{Bojan Mohar\thanks{Supported in part by an NSERC Discovery Grant (Canada),
  by the Canada Research Chair program, and by the
  Research Grant P1--0297 of ARRS (Slovenia).}~\thanks{On leave from:
  IMFM \& FMF, Department of Mathematics, University of Ljubljana, Ljubljana,
  Slovenia.}\\
  {Department of Mathematics}\\
  {Simon Fraser University}\\
  {Burnaby, B.C. V5A 1S6} \\
  email: {\tt mohar@sfu.ca}
}

\newcommand\CC{\hbox{$\cal{C}$}}
\newcommand\g{\hbox{\bf g}}  
\newcommand\sg{\overline{\hbox{\bf {g}}}}  

\newcommand{\DEF}[1]{{\em #1\/}}

\newcommand{\eopf}{\framebox[3mm]}
\newenvironment{proof}%
{\noindent{\bf Proof.}\ }%
{\hfill\eopf\par\bigskip}%
\begin{document}

\date{}

\maketitle

\begin{abstract}
A ``folklore conjecture, probably due to Tutte'' (as described in 
[P.D. Seymour, Sums of circuits, Graph theory and related topics 
(Proc. Conf., Univ. Waterloo, 1977), pp.~341--355, Academic Press, 1979])
asserts that every bridgeless cubic graph can be embedded on a surface 
of its own genus in such a way that the face boundaries are cycles
of the graph. Sporadic counterexamples to this conjecture have been known since 
the late 1970's. In this paper we consider closed 2-cell embeddings of graphs
and show that certain (cubic) graphs (of any fixed genus) 
have closed 2-cell embedding only in surfaces whose genus is very large
(proportional to the order of these graphs), thus providing a plethora of 
strong counterexamples to the above conjecture. 
The main result yielding such counterexamples may be of independent interest.
\end{abstract}


\section{Introduction}

In his seminal work {\em Sums of circuits\/} \cite{Se}, Paul Seymour stated the following
conjecture, which he has addressed as a ``folklore conjecture, probably due to
Tutte.''

\begin{conjecture}
\label{conj:1}
Any bridgeless cubic graph can be embedded on a surface of its own genus 
in such a way that the perimeters of all regions are circuits.
\end{conjecture}

As much as this conjecture was folklore in the nineteen seventies, 
today's methods of topological graph theory enable rather easy constructions of
counterexamples. There is evidence that several people were aware that 
Conjecture \ref{conj:1} is false for toroidal graphs. 
Apparently, an example of a toroidal cubic graph disproving Conjecture \ref{conj:1}
appears in the Ph.D. Thesis of Xuong \cite{Xu}. 
Richter \cite{Ri} found further examples of 2-connected (but not 3-connected)
cubic graphs of genus one for which the conjecture fails. 
Zha \cite{Zh3} constructed
graphs $G_g$, for each integer $g\ge1$, whose genus is $g$, but every embedding
in the orientable surface of genus $g$ has a face that is not bounded by a cycle
of the graph. (The examples of Zha are not cubic graphs, though.)
It is also stated by Zha in \cite{Zh3} that Archdeacon and Stahl, and Huneke,
Richter, and Younger (respectively) informed him of having constructed
further examples of toroidal cubic graphs disproving Conjecture \ref{conj:1}.
One purpose of this paper is to bring a rich new family of counterexamples
to the attention of interested graph theorists.

Our real goal is to provide simple examples, yet powerful enough to exhibit some
additional extremal properties. 
In particular, we shall consider (cubic) graphs of genus one.
If a cubic graph $G$ has an embedding in the torus with a face whose boundary is
not a cycle, then $G$ contains an edge $e$ whose removal yields a planar graph.
We call such a graph \DEF{near-planar} and refer to the edge $e$ as 
a \DEF{planarizing edge} (see \cite{CM1,CM2}). Our main result, 
Theorem \ref{thm:1}, gives a simple recipe for constructing near-planar
(cubic) graphs, all of whose embeddings of small genus have facial
walks that are not cycles. See Corollary \ref{cor:1} for more details.
Theorem \ref{thm:2} generalizes Theorem \ref{thm:1} to arbitrary surfaces.

An embedding of a graph in a surface is \DEF{strong} (sometimes also referred to
as a \DEF{closed 2-cell embedding} or a \DEF{circular embedding} \cite{RSS}) 
if each face boundary is a cycle in the graph.
We denote by $\g(G)$ the (orientable) \DEF{genus} of $G$. 
By $\sg(G)$ we denote the \DEF{strong genus} of $G$, which is defined as the 
smallest genus of an orientable surface in which $G$ has a strong embedding. 
If $G$ has no strong orientable embeddings, then $\sg(G)=\infty$.
In this notation, Conjecture \ref{conj:1} claims that $\sg(G)=\g(G)$ for every
bridgeless cubic graph $G$. 

The well-known \DEF{Cycle Double Cover Conjecture} claims that every 
2-edge-connected graph admits a collection of cycles such that every edge is
contained in precisely two of the cycles from the collection. Such a collection
is called a \DEF{cycle double cover} of the graph. If $G$ is cubic, then every
cycle double cover of $G$ determines a strong embedding of $G$ in some
surface (possibly non-orientable). There is also an orientable version of 
the cycle double cover conjecture:

\begin{conjecture}[Jaeger \cite{Ja85}]
\label{conj:oriCDC}
Every 2-connected graph $G$ has a strong embedding in some orientable surface,
i.e. $\sg(G)<\infty$.
\end{conjecture}

We follow standard graph theory terminology (see, e.g.~\cite{Di}). For the notions of
topological graph theory we refer to \cite{MT}.
All embeddings of graphs in surfaces are assumed to be \DEF{2-cell embeddings}.
If $S$ is a closed surface, whose Euler characteristic is $c=\chi(S)$,
then the \DEF{genus} of $S$ is equal to $\tfrac{1}{2}(2-c)$ if $S$ is
orientable, and is equal to $2-c$ if $S$ is non-orientable.

\section{Facial distance and nonseparating cycles}

Let $G$ be a graph embedded in a surface $S$ and let $x,y\in V(G)$, $x\ne y$.
We define the \DEF{facial distance} $d'(x,y)$ between $x$ and $y$ as the minimum 
integer $r$ such that there exist facial walks $F_1,\dots, F_r$ where $x\in V(F_1)$,
$y\in V(F_r)$, and $V(F_i)\cap V(F_{i+1})\ne\emptyset$ for every $i$, $1\le i<r$.
The following dual expression for $d'(x,y)$, see \cite{CM2}, can be viewed as 
a surface version of Menger's Theorem. 

\begin{proposition}
\label{prop:1}
Let $G$ be a plane graph and $x,y\in V(G)$, where $y$ lies on the boundary 
of the exterior face. Let $r$ be the maximum number of vertex-disjoint cycles,
$Q_1,\dots,Q_r$, contained in $G-x-y$, such that for $i=1,\dots,r$, 
$x\in int(Q_i)$ and $y\in ext(Q_i)$. Then $d'(x,y)=r+1$.
\end{proposition}

Let $\CC$ be a non-empty set of disjoint cycles of a graph $G$ that is
embedded in some surface, and let 
$C\subseteq G$ be the union of all cycles from \CC. Then \CC\ is 
\DEF{surface-separating} if there is a set of facial walks whose \DEF{sum}
(i.e., the symmetric difference of their edge-sets) is equal to $C$.
A set $\CC$ of disjoint cycles in an embedded graph $G$ is \DEF{homologically
independent} if no non-empty subset of $\CC$ is surface-separating.
We say that \CC\ is \DEF{induced and nonseparating} in the graph $G$ if
$C$ is an induced subgraph of $G$ and $G-V(C)$ is connected.

Induced and nonseparating cycles play a special role in
topological graph theory. In particular, if $G$ is a 3-connected planar graph,
then the induced and nonseparating cycles of $G$ are precisely those cycles that
form face boundaries. The following property generalizes this fact.

\begin{lemma}
\label{lem:facial or non-separating}
Let\/ \CC\ be a family of disjoint cycles in a graph $G$ such that
$C=\bigcup\CC$ is an induced and nonseparating subgraph of\/ $G$.
If\/ $G$ is embedded in some surface and no cycle in \CC\ is facial, then
$\CC$ is homologically independent.
\end{lemma}

The next lemma is taken from \cite{Mo92}. For completeness we include its proof.

\begin{lemma}
\label{lem:homologically independent}
Let\/ $G$ be a graph embedded in a surface of genus $g$ (either orientable, or 
non-orienta\-ble). If $Q_1,\ldots,Q_k$ are pairwise disjoint cycles in $G$ that
are homologically independent, then $k\leq g$. 
\end{lemma}

\begin{proof}
Since the cycles $Q_i$ are homologically independent,
it follows that after cutting the surface $\Sigma$ along these cycles,
we obtain a connected surface with boundary, having $2k$ boundary components 
if $\Sigma$ is orientable and having at least $k$ boundary components
in the non-orientable case. Denote their number by 
$b$. If we paste a disc on each of the boundary components we get a closed 
surface $\Sigma'$ whose Euler characteristic is equal to 
$\chi(\Sigma')=\chi(\Sigma)+b\leq 2$. If $\Sigma$ is orientable then
$2k = b\leq 2-\chi(\Sigma)=2-(2-2g)=2g$. In the non-orientable case 
we have $k \leq b\leq 2-\chi(\Sigma)=g$, which proves the claimed 
inequality in either case.
\end{proof}

\section{Strong embeddings of near-planar graphs}

Conjecture \ref{conj:1} holds for planar graphs. Namely, every bridgeless cubic
graph is 2-connected, and every embedding of a 2-connected graph in the plane is 
strong. This is no longer true on the torus. Figure \ref{fig:1} shows two 
embeddings of a non-planar cubic graph in the torus. One is strong, and the other
one is not.

\begin{figure}
\centering
\includegraphics[width=0.7\textwidth]{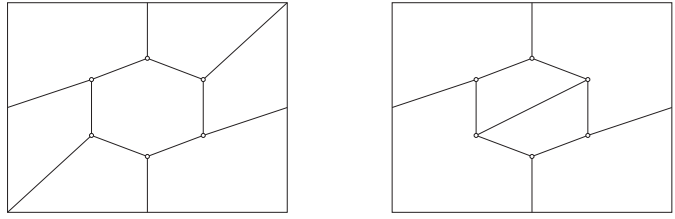}
\caption{$K_{3,3}$ has strong and non-strong embeddings in the torus}
 \label{fig:1}
\end{figure}

As mentioned in the introduction, a cubic toroidal graph admits a non-strong embedding 
in the torus if and only if it is near-planar. The following result gives rise
to a variety of counterexamples to Conjecture \ref{conj:1}.

\begin{theorem}
\label{thm:1}
Let\/ $G$ be a near-planar graph with $xy\in E(G)$ being a planarizing edge.
Suppose that $G-xy$ is a subdivision of a 3-connected graph and let $q=d'(x,y)$
be the facial distance between $x$ and $y$ in the (unique) planar embedding of\/
$G-xy$. Then every strong embedding of $G$ (either in an orientable or in 
a non-orientable surface) has genus at least $\lfloor \tfrac{1}{3}q\rfloor$.
\end{theorem}

\begin{proof}
Let $Q_1,\dots, Q_r$ ($r=q-1$) be the cycles guaranteed by Proposition \ref{prop:1},
enumerated such that $Q_i$ lies in the interior of $Q_{i+1}$ for $1\le i<r$.
Suppose that $G$ has a strong embedding $\Pi$ in a surface of genus 
$g$, and consider a facial cycle $F$ containing the edge $xy$. 
Note that $F$ intersects all cycles $Q_1,\dots, Q_r$ and hence contains,
for each $i\in \{1,\dots,r-1\}$, a path $R_i$ joining $Q_i$ and $Q_{i+1}$.
Since every edge belongs to two facial cycles only,
this implies that there is a facial cycle $F_i$ in the planar embedding of
$G-xy$ that contains an edge in $R_i$ and is not $\Pi$-facial. 
Since $G-xy$ is a subdivided 3-connected graph, the cycle $F_i$ is induced
and nonseparating. By Lemma \ref{lem:facial or non-separating}, $F_i$ is 
surface-non-separating under the embedding~$\Pi$.

Let us now consider the cycles $F_1, F_4, F_7, \dots, F_{3k-2}$,
where $k = \lfloor (r+1)/3\rfloor$. Since each $F_i$ is induced and 
$F_1,F_4, F_7,\dots$ are at distance at least two from each other,
the union $C=F_1 \cup F_4 \cup \cdots \cup F_{3k-2}$ of these cycles
is an induced subgraph of $G$. Next, we argue that $C$ is 
nonseparating in $G$. To see that, let $u\in V(G-C)$. There is an index
$i\in\{0,1,\dots,r\}$ such that $u\in int(Q_{i+1})\cap ext(Q_i)$
(where $Q_0=\{x\}$, $Q_{r+1}=\{y\}$, and $ext(Q_0) = int(Q_{r+1}) = G$).
Since $F_i$ is nonseparating in $G$, there is a path from $u$ to
$Q_i$ in $G-C$. Similarly, for each $j=0,1,\dots,r$, there is a path
in $G-C$ from $Q_j$ to $Q_{j+1})$. This easily implies that $G-C$ is connected
and proves that $C$ is nonseparating. 

Since $C$ is an induced and nonseparating subgraph of $G$ and none of the
cycles forming $C$ is $\Pi$-facial, Lemma \ref{lem:facial or non-separating} 
shows that the cycles forming $C$ are homologically independent. By Lemma 
\ref{lem:homologically independent}, we conclude that $k\le g$, 
which we were to prove.
\end{proof}

\begin{figure}
\centering
\includegraphics[width=0.4\textwidth]{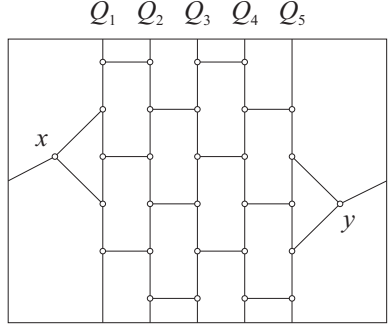}
\caption{A near-planar graph $G_2$}
 \label{fig:2}
\end{figure}

\begin{corollary}
\label{cor:1}
For every integer $n$, there exists a near-planar cubic graph\/ $G_n$ of order\/
$18n-6$ with\/ $\sg(G)\ge n$.
\end{corollary}

\begin{proof}
Let $G_n$ be the near-planar graph whose toroidal embedding is shown in Figure 
\ref{fig:2} (for $n=2$). The graph has $r=3n-1$ disjoint cycles that show that
$d'(x,y) = r+1$ in the planar embedding of $G_n-xy$. Now, Theorem \ref{thm:1} applies.
\end{proof}

Let us observe that Euler's formula implies that {\em every} orientable
(2-cell) embedding of the graphs $G_n$ from Corollary \ref{cor:1}
has genus at most $6n-2$. The largest possible genus of a strong embedding
of $G_n$ is $6n-3$, and this bound is attained if and only if $G_n$ has
a cycle double cover consisting of three Hamilton cycles of $G_n$.

\section{Examples of higher genus}

For embeddings in surfaces of higher genus, we need an additional 
notion. An embedding of a graph is \DEF{polyhedral} if every facial 
walk is an induced and nonseparating cycle of the graph.


\begin{theorem}
\label{thm:2}
Let\/ $G$ be a graph  and $xy\in E(G)$. Suppose that $G-xy$ has
a polyhedral embedding in some surface, and let $q=d'(x,y)$
be the facial distance between $x$ and $y$ under this embedding. 
Then every strong embedding of $G$ has genus at least 
$\lfloor \tfrac{1}{3}q\rfloor$.
\end{theorem}

\begin{proof}
The proof is essentially the same as the proof of Theorem \ref{thm:1}.
Let us only remark the main differences. First of all, there are pairwise disjoint
subgraphs $Q_1,\dots,Q_r$ ($r=q-1$), where each $Q_i$ is a union of cycles (i.e. an
Eulerian subgraph) of $G-x-y$ that separates the surface so that one part 
(called the \DEF{interior} of $Q_i$, $int(Q_i)$) contains $x$ and the
previous cycles $Q_1,\dots,Q_{i-1}$, and the other part of the surface 
(the \DEF{exterior} $ext(Q_i)$) contains $y$ and $Q_{i+1},\dots,Q_r$. 
These subgraphs $Q_i$ are obtained as follows. 
We set $Q_0=\{x\}$. For $i\le r$, having constructed $Q_{i-1}$, we take
$int(Q_{i-1})$ together with all facial cycles that intersect $Q_{i-1}$
and denote this subgraph of $G$ by $int(Q_i)$. The sum of all facial cycles
forming $int(Q_i)$ is a surface-separating subgraph $Q_i$ of $G$.
(If $G$ is a cubic graph, then $Q_i$ is disjoint union of one or more cycles,
but in general, it is just an Eulerian subgraph of $G$.)

The rest of the proof is the same as for Theorem~\ref{thm:1}.
\end{proof}

\section{Concluding remarks}

In view of the Cycle double cover conjecture and its embedding counterpart
(Conjecture \ref{conj:oriCDC}), it is of interest to show that (cubic) graphs
of small genus admit strong embeddings in some surface. Some existing work
in this area is \cite{Z1,Z2}. Let us remark that for cubic graphs this follows 
from known results about possible counterexamples to the Cycle double cover
conjecture. Goddyn \cite{Go} proved that a minimum counterexample has
girth at least 10, and Huck \cite{Hu} extended this further by proving that
its girth is at least 12. The reductions used in those proofs can be made
so that embeddability in a fixed surface is preserved. They can be summarized
as follows.

\begin{theorem}
\label{thm:Huck}
Let $g\ge0$ be an integer. If there is a 2-edge-connected graph whose genus 
(or non-orientable genus) is at most $g$ and that does not have a cycle double
cover, then there is such a graph $G$ with the following properties:
\begin{itemize}
\item[\rm (a)] $G$ is cubic and\/ $3$-connected.
\item[\rm (b)] $G$ has girth at least 12.
\end{itemize}
\end{theorem}

This yields the following straightforward corollary.

\begin{corollary}
\label{cor:2}
If a 2-edge-connected graph $G$ has genus at most 16 
or has nonorientable genus at most 33,
then it has a cycle double cover. In particular, if it is cubic, then it
admits a strong embedding in some surface.
\end{corollary}

\begin{proof}
By Theorem \ref{thm:Huck}, we may assume that $G$ is cubic and has girth at 
least 12. The girth condition implies that the order $n=|V(G)|$ of $G$ satisfies
$$
   n \ge 1+3+6+12+24+48+32 = 126.
$$
Since $G$ is cubic, $|E(G)|= \tfrac{3}{2}n$. Since every facial walk has length at
least 12, we have that $12f\le 2|E(G)| = 3n$, where $f$ denotes the number of facial
walks. By Euler's formula, $f=c-n + |E(G)| \ge \tfrac{1}{2}n +c$, where $c$ is the
Euler characteristic of the surface in which $G$ is embedded. This implies that
$-c\ge \tfrac{1}{2}n - f \ge \tfrac{1}{4}n\ge \tfrac{63}{2}$. In particular, 
the genus (or the nonorientable genus) is at least 17 (respectively, 34).
\end{proof}

The proofs of Goddyn \cite{Go} and Huck \cite{Hu} do not
preserve orientability of embeddings, so it remains an open problem if the last
conclusion of Corollary \ref{cor:2} can be strengthened by concluding
that a strong embedding in some orientable surface exists.
A recent paper by Ellingham and Zha \cite{EZ} discusses this problem
and provides a solution for projective-planar graphs.

The following computational problems are of interest.

\begin{problem}
\label{prb:1}
What is the computational complexity of determining if a given (cubic) graph admits 
a strong embedding in some (orientable) surface?
\end{problem}

If Conjecture \ref{conj:oriCDC} is true, then the answer to Problem \ref{prb:1}
is trivial---such an embedding exists if and only if the input graph is
2-edge-connected. On the other hand, if Conjecture \ref{conj:oriCDC} fails, it is
likely that the problem of existence of strong embeddings would be NP-hard.

In this note we have made a small step towards the study of the following
problems.

\begin{problem}
\label{prb:2}
For a given input graph $G$ that has at least one strong embedding, 
find a strong embedding of minimum genus.
\end{problem}

\begin{problem}
\label{prb:3}
For a fixed integer $g\ge0$ decide if a given input graph $G$, whose genus is at
most $g$, admits a closed 2-cell embedding of genus at most~$g$.
\end{problem}

This problem is trivial for $g=0$, and the methods of this paper may provide
a way to a fast recognition for the case $g=1$.

\end{document}